\documentclass{article}

\usepackage{amsmath}
\usepackage{amsthm}
\usepackage{alltt}
\usepackage{hyperref}
\usepackage{algorithm}
\usepackage{algpseudocode}
\usepackage{graphicx}

\numberwithin{equation}{section}

\newtheorem{theorem}{Theorem}[section]

\newcommand{\gaussian}[2]
{\genfrac{(}{)}{0pt}{}{#1}{#2}_{\textstyle q}}
\newcommand{\gaussianrec}[2]
{\genfrac{(}{)}{0pt}{}{#1}{#2}_{\textstyle q^{-1}}}

\title{The q-Binomial Coefficient for Negative Arguments\\
 and Some q-Binomial Summation Identities}
\author{M.J. Kronenburg}
\date{}

\begin{document}

\maketitle

\begin{abstract}
Using a property of the q-shifted factorial, an identity for
q-binomial coefficients is proved, which is used to derive
the formulas for the q-binomial coefficient for negative arguments.
The result is in agreement with an earlier paper
about the normal binomial coefficient for negative arguments.
Some new q-binomial summation identities are derived,
and the formulas for negative arguments transform some of these
summation identities into each other.
One q-binomial summation identity is transformed
into a new q-binomial summation identity.
\end{abstract}


\section{Definitions and Basic Identities}

Let the following definition of the q-binomial coefficient,
also called the Gaussian polynomial, be given.\\
\begin{equation}\label{gaussdef}
 \gaussian{m+p}{m} = \prod_{j=1}^m \frac{1-q^{p+j}}{1-q^j}
\end{equation}
Let the q-shifted factorial, also called the q-Pochhammer symbol, be given by \cite{NT}:
\begin{equation}
 (a;q)_k = \prod_{j=0}^{k-1} (1-aq^j)
\end{equation}
Then the q-binomial coefficient for integer $k\geq 0$ is:
\begin{equation}\label{gaussfac}
 \gaussian{n}{k} = \prod_{j=1}^k \frac{1-q^{n-k+j}}{1-q^j} = \frac{(q^{n-k+1};q)_k}{(q;q)_k}
\end{equation}
Let $\Gamma_q(x)$ be the q-gamma function \cite{AAR}.\\
For complex $x$, $y$:
\begin{equation}\label{qbindef}
 \gaussian{x}{y} = \frac{\Gamma_q(x+1)}{\Gamma_q(y+1)\Gamma_q(x-y+1)}
\end{equation}
From this follows the symmetry identity:\\
For complex $x$, $y$:
\begin{equation}\label{gaussymdef}
 \gaussian{x}{y} = \gaussian{x}{x-y}
\end{equation}
The functional equation of the $\Gamma_q(x)$ function is \cite{AAR}:
\begin{equation}\label{gamfun}
 \Gamma_q(x+1) = \frac{1-q^x}{1-q} \Gamma_q(x)
\end{equation}
Combination of (\ref{qbindef}) and (\ref{gamfun}) gives the absorption identity:\\
For complex $x$, $y$:
\begin{equation}\label{absorbe}
 \gaussian{x}{y} = \frac{1-q^x}{1-q^y} \gaussian{x-1}{y-1}
\end{equation}
From definition (\ref{gaussfac}) follows:\\
For integer $n\geq 0$ and integer $k$:
\begin{equation}\label{zero1}
 \gaussian{n}{k} = 0 \textrm{~if~} k>n
\end{equation}
From this and (\ref{gaussymdef}) follows:\\
For integer $n\geq 0$ and integer $k$:
\begin{equation}\label{zero2}
 \gaussian{n}{k} = 0 \textrm{~if~} k<0
\end{equation}

\section{The q-Binomial Coefficient for Negative Arguments}

For deriving the q-binomial coefficient for negative arguments,
the following theorem \cite{KS,NT} is needed.
\begin{theorem}
For integer $k\geq 0$:
\begin{equation}
 (a;q)_k = (-a)^k q^{k(k-1)/2} (\frac{q^{1-k}}{a};q)_k
\end{equation}
\end{theorem}
\begin{proof}
From:
\begin{equation}
 k(k-1)/2 = \sum_{j=0}^{k-1} j = \sum_{j=0}^{k-1} (k-j-1)
\end{equation}
follows:
\begin{equation}
 q^{k(k-1)/2} = \prod_{j=0}^{k-1} q^j = \prod_{j=0}^{k-1} q^{k-j-1}
\end{equation}
This is used in:
\begin{equation}
 (-a)^k q^{k(k-1)/2} \prod_{j=0}^{k-1} ( 1- \frac{q^{1-k}}{a} q^j ) 
 = q^{k(k-1)/2} \prod_{j=0}^{k-1} ( q^{1-k+j} - a) 
 = \prod_{j=0}^{k-1} ( 1 - a q^{k-j-1} ) = \prod_{j=0}^{k-1} ( 1 - a q^j )
\end{equation}
\end{proof}
\begin{theorem}\label{trans1}
For integer $k\geq 0$:
\begin{equation}
 \gaussian{n}{k} = (-1)^k q^{nk-k(k-1)/2} \gaussian{-n+k-1}{k}
\end{equation}
\end{theorem}
\begin{proof}
Using (\ref{gaussfac}) and the previous theorem and $-n=(-n+k-1)-k+1$:
\begin{equation}
\begin{split}
 & \gaussian{n}{k} = \frac{(q^{n-k+1};q)_k}{(q;q)_k} = (-1)^k q^{k(n-k+1)+k(k-1)/2}\frac{(q^{-n};q)_k}{(q;q)_k} \\
 & = (-1)^k q^{nk-k(k-1)/2} \gaussian{-n+k-1}{k} \\
\end{split}
\end{equation}
\end{proof}
\begin{theorem}\label{trans2}
For integer $k\leq n$:
\begin{equation}
 \gaussian{n}{k} = (-1)^{n-k} q^{(n-k)(n+k+1)/2} \gaussian{-k-1}{n-k}
\end{equation}
\end{theorem}
\begin{proof}
In the previous theorem replacing $k$ with $n-k$, which makes it valid when $n-k\geq 0$
which is when $k\leq n$, and then applying (\ref{gaussymdef}) to the left side:
\begin{equation}\label{gaussym}
 \gaussian{n}{n-k} = \gaussian{n}{k}
\end{equation}
and using $n(n-k)-(n-k)(n-k-1)/2=(n-k)(n+k+1)/2$
gives this theorem.
\end{proof}
These two transformations can be used to transform one q-binomial
summation identity into another,
and can be used to express the q-binomial coefficient
for negative integer $n$ and integer $k$
into the q-binomial coefficient for nonnegative integer $n$ and $k$.
\begin{theorem}\label{binomnegdef}
For negative integer $n$ and integer $k$:
\begin{equation}
 \gaussian{n}{k} =
 \begin{cases}
   \displaystyle (-1)^k q^{nk-k(k-1)/2} \gaussian{-n+k-1}{k} & \text{if $k\geq 0$} \\
   \displaystyle (-1)^{n-k} q^{(n-k)(n+k+1)/2} \gaussian{-k-1}{n-k} & \text{if $k\leq n$} \\
   0 & \text{otherwise} \\
 \end{cases}
\end{equation}
\end{theorem}
\begin{proof}
The first two cases are identical to theorems \ref{trans1} and \ref{trans2}.\\
For the third case, from (\ref{absorbe}) follows for integer $n$, $k$:
\begin{equation}
 \gaussian{n-1}{k-1} = \frac{1-q^k}{1-q^n} \gaussian{n}{k}
\end{equation}
When $k=0$ the right side is zero,
so when $n<0$ and $k=0$ this identity produces zeros for $n<k<0$,
which is the third case.
This identity does not produce zeros for all $k<0$
because when $n>0$ a point will be reached where $n=0$ and this expression becomes
$(1-q^k)0/0$ which is undefined.
\end{proof}
The normal binomial coefficients are the q-binomial coefficients with $q=1$,
in which case this theorem reduces to theorem 2.1 in \cite{K15}.\\
For $n=-1$ this theorem results in:
\begin{equation}
 \gaussian{-1}{k} =
\begin{cases}
 (-1)^k q^{-k(k+1)/2} & \textrm{\rm~if~} k\geq 0 \\
 (-1)^{k+1} q^{-k(k+1)/2} & \textrm{\rm~if~} k\leq -1 \\
\end{cases}
\end{equation}
which is in agreement with example 1.4 in \cite{FS}.\\
As an example of the second case of this theorem:
\begin{equation}
 \gaussian{-3}{-5} = q^{-7} \gaussian{4}{2} = q^{-7} (1+q^2)(1+q+q^2)
\end{equation}
which is in agreement with example 1.2 in \cite{FS}.\\
The q-binomial coefficient polynomial is palindromic, 
which means that $a_j=a_{k(n-k)-j}$,
from which follows that it is self-reciprocal,
where $j$ is replaced by $k(n-k)-j$:
\begin{equation}\label{selfrec}
\begin{split}
 & \gaussian{n}{k} = \sum_{j=0}^{k(n-k)} a_j q^j = \sum_{j=0}^{k(n-k)} a_{k(n-k)-j} q^j 
 = \sum_{j=0}^{k(n-k)} a_j q^{k(n-k)-j} \\
 & = q^{k(n-k)} \sum_{j=0}^{k(n-k)} a_j q^{-j} = q^{k(n-k)} \gaussianrec{n}{k} \\
\end{split}
\end{equation}
The theorems above leave all q-binomial coefficients
self-reciprocal.\\
For theorem \ref{trans1}:
\begin{equation}
 q^{k(n-k)} \gaussianrec{n}{k} = (-1)^k q^{nk-k(k-1)/2} q^{-k(n+1)} \gaussianrec{-n+k-1}{k}
\end{equation}
and because $-k(n-k)+nk-k(k-1)/2-k(n+1)=-(nk-k(k-1)/2)$:
\begin{equation}
 \gaussianrec{n}{k} = (-1)^k q^{-(nk-k(k-1)/2)} \gaussianrec{-n+k-1}{k}
\end{equation}
For theorem \ref{trans2}:
\begin{equation}
 q^{k(n-k)} \gaussianrec{n}{k} = (-1)^{n-k} q^{(n-k)(n+k+1)/2} q^{-(n-k)(n+1)} \gaussianrec{-k-1}{n-k}
\end{equation}
and because $-k(n-k)+(n-k)(n+k+1)/2-(n-k)(n+1)=-(n-k)(n+k+1)/2$:
\begin{equation}
 \gaussianrec{n}{k} = (-1)^{n-k} q^{-(n-k)(n+k+1)/2} \gaussianrec{-k-1}{n-k}
\end{equation}
\newpage 

\section{Some q-Binomial Summation Identities}

Some q-binomial summation identities are derived and it is shown
how q-binomial coefficients with negative arguments transform
one summation identity into another.
The following identities are the q-binomial theorem and the q-binomial
theorem for negative powers \cite{wiki2}:
\begin{equation}\label{qbinpos}
 \prod_{k=0}^{n-1} (1+xq^k) = \sum_{k=0}^n q^{\textstyle\binom{k}{2}} \gaussian{n}{k} x^k
\end{equation}
\begin{equation}\label{qbinneg}
 \frac{1}{\prod_{k=0}^{n-1}(1-xq^k)} = \sum_{k=0}^{\infty} \gaussian{n+k-1}{n-1} x^k
\end{equation}
The following is an obvious product rule:
\begin{equation}\label{qbinprod}
 \prod_{k=0}^{a-1} (1+xq^k) \prod_{k=0}^{b-1} (1+xq^{a+k}) = \prod_{k=0}^{a+b-1} (1+xq^k)
\end{equation}
With these three identities some q-binomial summation identities are derived,
using that the coefficients of a product of two polynomials are the convolutions
of the coefficients of the two polynomials.
\begin{theorem}\label{qbinsum1}
The q-analog of the Chu-Vandermonde identity \cite{A84}:
\begin{equation}
 \sum_{k=0}^n q^{(a-k)(n-k)} \gaussian{a}{k} \gaussian{b}{n-k} = \gaussian{a+b}{n}
\end{equation}
\end{theorem}
\begin{proof}
Using (\ref{qbinpos}) with (\ref{qbinprod}):
\begin{equation}
 ( \sum_{k=0}^a q^{\textstyle\binom{k}{2}} \gaussian{a}{k} x^k )
 ( \sum_{k=0}^b q^{\textstyle\binom{k}{2}} \gaussian{b}{k} q^{ak} x^k )
 = \sum_{k=0}^{a+b} q^{\textstyle\binom{k}{2}} \gaussian{a+b}{k} x^k
\end{equation}
The coefficients of both sides must be equal:
\begin{equation}
 \sum_{k=0}^n q^{\textstyle\binom{k}{2}} \gaussian{a}{k} q^{\textstyle\binom{n-k}{2}} \gaussian{b}{n-k} q^{a(n-k)}
 = q^{\textstyle\binom{n}{2}} \gaussian{a+b}{n}
\end{equation}
Because:
\begin{equation}
 \binom{k}{2} + \binom{n-k}{2} + a(n-k) - \binom{n}{2} = (a-k)(n-k)
\end{equation}
the theorem is proved.
By replacing $k$ with $n-k$ and interchanging $a$ and $b$, the power of $q$ in the summand
can be replaced by $q^{(b-n+k)k}$.
\end{proof}
\begin{theorem}\label{qbinsum2}
\begin{equation}
 \sum_{k=0}^n q^{(b+1)k} \gaussian{a+k}{a} \gaussian{b+n-k}{b} = \gaussian{n+a+b+1}{n}
\end{equation}
\end{theorem}
\begin{proof}
Using the reciprocal of (\ref{qbinprod}) with $-x$:
\begin{equation}
 \frac{1}{\prod_{k=0}^{a-1} (1-xq^k)} \frac{1}{\prod_{k=0}^{b-1} (1-xq^{a+k})} 
 = \frac{1}{\prod_{k=0}^{a+b-1} (1-xq^k)}
\end{equation}
which with (\ref{qbinneg}) becomes:
\begin{equation}
 ( \sum_{k=0}^{\infty} \gaussian{a+k-1}{a-1} x^k )
 ( \sum_{k=0}^{\infty} \gaussian{b+k-1}{b-1} q^{ak} x^k )
 = \sum_{k=0}^{\infty} \gaussian{a+b+k-1}{a+b-1} x^k
\end{equation}
The coefficients of both sides must be equal:
\begin{equation}
 \sum_{k=0}^n \gaussian{a+k-1}{a-1} \gaussian{b+n-k-1}{b-1} q^{a(n-k)} = \gaussian{a+b+n-1}{a+b-1} = \gaussian{n+a+b-1}{n}
\end{equation}
Replacing $a$ by $a+1$ and $b$ by $b+1$, and then replacing $k$ by $n-k$ and interchanging
$a$ and $b$ gives the theorem.
By replacing $k$ with $n-k$ and interchanging $a$ and $b$, the power of $q$ in the summand
can be replaced by $q^{(a+1)(n-k)}$.
\end{proof}
\begin{theorem}\label{qbinsum3}
\begin{equation}
 \sum_{k=0}^n (-1)^k q^{\textstyle\binom{k}{2}} \gaussian{a}{k} \gaussian{b+n-k}{b}
 = 
\begin{cases}
\displaystyle q^{an} \gaussian{n-a+b}{n} & \textrm{\rm ~if~} a\leq b \\
\displaystyle (-1)^n q^{bn+n(n+1)/2} \gaussian{a-b-1}{n} & \textrm{\rm ~if~} a>b \\
\end{cases}
\end{equation}
\end{theorem}
\begin{proof}
From (\ref{qbinprod}) replacing $a$ with $b$ and $b$ with $a-b$ gives:
\begin{equation}
 \frac{\prod_{k=0}^{a-1}(1+xq^k)}{\prod_{k=0}^{b-1}(1-(-x)q^k)} = \prod_{k=0}^{a-b-1}(1+xq^{b+k})
\end{equation}
Therefore:
\begin{equation}
 (\sum_{k=0}^a q^{\textstyle\binom{k}{2}} \gaussian{a}{k} x^k)
 (\sum_{k=0}^{\infty} \gaussian{b+k-1}{b-1} (-1)^k x^k)
 = \sum_{k=0}^{a-b} q^{\textstyle\binom{k}{2}} \gaussian{a-b}{k} q^{bk} x^k
\end{equation}
The coefficients of both sides must be equal:
\begin{equation}
 \sum_{k=0}^n q^{\textstyle\binom{k}{2}} \gaussian{a}{k} \gaussian{b+n-k-1}{b-1} (-1)^{n-k}
 = q^{bn+\textstyle\binom{n}{2}} \gaussian{a-b}{n}
\end{equation}
Replacing $b$ by $b+1$ gives the second case of the theorem.
When $a\leq b$ application of theorem \ref{trans1} 
to the right side of the second case gives the first case of the theorem.
\end{proof}
The special cases $b=a-1$ and $a=n$, $b=0$ of this theorem appear in exercise 3.9 in \cite{NT}.

\section{Transforming q-Binomial Summation Identities}

From theorem \ref{qbinsum2} and theorem \ref{trans2} the following
theorem follows.
\begin{theorem}\label{qbinsum4}
\begin{equation}
 \sum_{k=0}^n (-1)^k q^{(a-b)(n-k)+k(k+1)/2} \gaussian{a}{k}\gaussian{b+n-k}{b} =
\begin{cases}
 \displaystyle q^{(a-b)n} \gaussian{n-a+b}{n} & \textrm{\rm ~if~} a\leq b \\
 \displaystyle (-1)^n q^{n(n+1)/2} \gaussian{a-b-1}{n} & \textrm{\rm ~if~} a > b \\
\end{cases}
\end{equation}
\end{theorem}
\begin{proof}
In theorem \ref{qbinsum2} replacing $a$ by $-a$ and using theorem \ref{trans2}:
\begin{equation}
 \gaussian{-a+k}{-a} = (-1)^k q^{-ak+k(k+1)/2} \gaussian{a-1}{k}
\end{equation}
and replacing $a$ by $a+1$ gives the first case of this theorem.
For the second case of the theorem, using theorem \ref{trans1} when $a>b$:
\begin{equation}
 \gaussian{n-a+b}{n} = (-1)^n q^{(n-a+b)n-n(n-1)/2} \gaussian{a-b-1}{n}
\end{equation}
gives the second case of the theorem.
This theorem is identical to theorem \ref{qbinsum3} except for the power of $q$ in the summand.
\end{proof}
Some summation identities from the previous section can be transformed into each other.
Theorem \ref{qbinsum3} can be transformed into theorem \ref{qbinsum2} using theorem \ref{trans1}
by replacing $a$ with $-a$:
\begin{equation}
 \gaussian{-a}{k} = (-1)^k q^{-ak-\textstyle\binom{k}{2}} \gaussian{a+k-1}{a-1}
\end{equation}
which with the first case of theorem \ref{qbinsum3} gives:
\begin{equation}
 \sum_{k=0}^n q^{a(n-k)} \gaussian{a+k-1}{a-1} \gaussian{b+n-k}{b} = \gaussian{n+a+b}{n}
\end{equation}
Replacing $a$ with $a+1$ and $k$ with $n-k$ and interchanging $a$ and $b$ gives theorem \ref{qbinsum2}.\\
Theorem \ref{qbinsum3} can be transformed into theorem \ref{qbinsum1} using theorem \ref{trans2}
by replacing $b$ with $-b$:
\begin{equation}
 \gaussian{-b+n-k}{-b} = (-1)^{n-k} q^{(n-k)(-2b+n-k+1)/2} \gaussian{b-1}{n-k} 
\end{equation}
which with the second case of theorem \ref{qbinsum3} by replacing $b$ by $b+1$ and using\\
$k(k-1)/2+(n-k)(-2(b+1)+n-k+1)/2+bn-n(n-1)/2=k(b-n+k)$
gives:
\begin{equation}
 \sum_{k=0}^n q^{k(b-n+k)} \gaussian{a}{k}\gaussian{b}{n-k} = \gaussian{a+b}{n}
\end{equation}
Replacing $k$ by $n-k$ and interchanging $a$ and $b$ gives theorem \ref{qbinsum1}.\\

\pdfbookmark[0]{References}{}

\end{document}